\documentclass[a4paper, 12pt, notitlepage]{article}
\usepackage{booktabs}
\usepackage[utf8]{inputenc}
\usepackage[english]{babel}
\usepackage{natbib}
\setcitestyle{authoryear,open={(},close={)}} %Citation-related commands
\usepackage{psfrag}
\usepackage{graphicx}
\usepackage{epsfig}
\usepackage{amssymb}
\usepackage{commath}
\usepackage{amsmath}
\usepackage{enumerate}
\usepackage[margin=1in]{geometry}
\usepackage{fancyhdr}
\usepackage{amsfonts}
\usepackage{color}                    % For creating coloured text and background
\usepackage{hyperref}
\usepackage{latexsym}
\usepackage{subfigure}
\newenvironment{proof}[1][Proof]{\textbf{#1.} }{\ \rule{0.5em}{0.5em}}

\newtheorem{theorem}{Theorem}

\newtheorem{remark}{Remark}

\title{Deterministic Epidemic Models For Ebola Infection With Time-dependent Controls.} % change this
\author{E. Okyere
\thanks{Department of Mathematics and Statistics, UENR, Sunyani, {\tt eric.okyere@uenr.edu.gh}}
\and J. D. Ankamah\thanks{Department of Mathematics and Statistics, UENR, Sunyani,{\tt johnson.ankamah@uenr.edu.gh}}\and A. K. Hunkpe \thanks{Department of Mathematics and Statistics, UENR, Sunyani, {\tt hunkpeanthonykodzo@gmail.com}}\and D. Mensah\thanks{Department of Mathematics and Statistics, UENR, Sunyani, {\tt mensahdorcas704@gmail.com}}}

\begin{document}
%\pagenumbering{roman}
%\tableofcontents
%%%%%%%%%% PRELIMINARY MATERIAL %%%%%%%%%%
\maketitle
\begin{abstract}
In this paper, we have studied epidemiological models for Ebola infection using nonlinear ordinary differential equations and optimal control theory. We considered optimal control analysis of SIR and SEIR models for the deadly Ebola infection using vaccination, treatment and educational campaign as time-dependent controls functions. We have applied indirect methods to study existing deterministic optimal control epidemic models for Ebola virus disease. These methods in optimal control are based on Hamiltonian function and the Pontryagin's maximum principle to construct adjoint equations and optimality systems. The forward-backward sweep numerical scheme with fourth-order Runge-Kutta method is used to solve the optimality system for the various control strategies. From our numerical illustrations, we can conclude that, effective educational campaigns and vaccination of susceptible individuals as were as effective treatments of infected individuals can help reduce the disease transmission.

%\keywords{Optimal Control \and Forward-backward Sweep Method \and Ebola Infections}
\end{abstract}

\section{Introduction}
The re-emergence of the Ebola virus disease in $2014-2016$ in West Africa has been classified as the largest outbreak since the disease was first discovered in DRC in 1976 \citep{WHO2019ebola}. This highly infectious and deadly disease has claimed many lives and caused huge economic burden in the affected West Africa Countries. The recent $2018-2019$ outbreak in eastern DRC is a very complex situation due to insecurities which is seriously affecting public health workers response activities \citep{WHO2019ebola}.\\

Mathematical modelling of epidemics has contributed significantly in understanding the dynamical behaviour and control of infectious diseases \citep{hethcote2000mathematics}. The complex dynamics of Ebola virus disease has attracted the attention of many researchers who are interested in epidemiological modeling [see, e.g, \cite{tulu2017mathematicaltoday, dong2015evaluationtoday, althaus2014estimatingtoday, chowell2014transmissiontoday, NgwaEBOLA2016, weitz2015modeling, ndanguza2017analysis, diaz2018modified, atangana2014mathematical, jiang2017mathematical, berge2017simple, xia2015modeling, khan2015estimating, chowell2004basic, goufo2016stability, agusto2015mathematical, LUO2019EBOLA, ABAGUMEL2019}].\\

System of nonlinear equations that incorporates optimal control dynamics are key mathematical tools that are used in compartmental modelling to understand the spread of infectious diseases. In \cite{sharomi2017optimal}, they conducted a comprehensive survey on optimal control modeling of several infectious diseases. \cite{zakary2017multi} formulated epidemic model for controlling Ebola outbreak. A mathematical model for Ebola disease with time-dependent controls is proposed and numerically analyzed in \citep{bonyah2016optimal}. The authors in \citep{ahmad2016optimal} proposed and studied optimal control epidemic model for Ebola virus infection. \cite{area2017ebola} introduced and studied deterministic epidemic model for Ebola disease that incorporates optimal control dynamics with vaccination control. The authors in \citep{grigorieva2017optimal} developed and analyzed controlled dynamical model for Ebola virus infection.\\

Our present work is motivated by indirect methods in optimal control theory. These methods are based on the Pontryagin's maximum principle and  Hamiltonian function to construct adjoint equations and optimality systems for optimal control problems. Indirect methods have widely been applied by several authors, see the detailed survey on optimal modeling of infectious diseases by the authors in \citep{sharomi2017optimal} and references therein. Recently, indirect methods in optimal control have also been used to study the dynamical behaviour and control of Zika virus disease [see, e.g, \cite{bonyah2017theoreticalBB, MKHAN2019_ZIKA, miyaoka2019optimal, BONYAH2019_ZIKA}]. We are also inspired by the forward-backward numerical scheme with fourth order Runge-Kutta method for optimal control problems described in \citep{lenhart2007optimal}. \\

In the works by the authors in \citep{rachah2015mathematical, rachah2016dynamics}, they formulated and analyzed SIR and SEIR optimal control models for Ebola infection using direct methods. In both papers, they formulated objective functionals with dynamical state constrained equations. In \citep{rachah2015mathematical}, they  proposed one optimal control strategy for the SIR model and the same authors in \citep{rachah2016dynamics} proposed three different optimal control strategies for the SEIR model. After their models formulations, they applied the ACADO solver which is an automatic control and dynamic optimization tool to perform their numerical simulations. In our present work, will apply indirect methods to  study the mathematical models proposed in \citep{rachah2015mathematical, rachah2016dynamics}. In this study, we will also propose one additional optimal control strategy for the SIR model. We will analytically construct Hamiltonian function and optimality system for the various optimal control strategies. We will then apply the forward-backward sweep method with fourth-order Runge-Kutta method to perform numerical simulations in Matlab.\\

The rest of the paper is organized as follows. In section~\ref{eric44}, we will introduce the classic SIR epidemic model to describe the dynamical behaviour of Ebola infection. We will then consider optimal control problems for the SIR model with two control strategies. In section~\ref{abena779}, we will present the basic SEIR epidemic model to describe the transmission dynamics of Ebola virus disease. We will further consider SEIR optimal control problems for this infectious disease with three different control strategies. In all these sections, we will formulate Hamiltonian functions and then apply the Pontryagin's maximum principle to construct adjoint equations and optimality system for the various optimal control strategies. We will also perform numerical simulations for the optimality systems. Finally, we will conclude the paper in section~\ref{threesec}.

\section{SIR Model}\label{eric44}
In this section, we introduce the classic SIR epidemic model to describe the dynamical behaviour of Ebola infection. The model assumes constant population size with no vital dynamics \citep{hethcote2000mathematics}. The population is divided into three different classes with $S(t),\ I(t)$ and $R(t)$ representing Susceptible, Infected and Recovered individuals respectively. As in \citep{rachah2015mathematical}, the nonlinear dynamical system describing Ebola infection is given by

\begin{equation}\label{anthony1}
\begin{cases}
\frac{dS(t)}{dt}={-\frac{\nu S(t)I(t)}{N}}, \ \ \qquad \qquad S(0) = S_0\geq 0, &\\ \\
\frac{dI(t)}{dt} =  {\frac{\nu S(t)I(t)}{N}}-{\delta}{I(t)}, \qquad  I(0) = I_0\geq 0,\\ \\
\frac{dR(t)}{dt} = {\delta I(t)}, \ \qquad \qquad \qquad R(0) = R_0\geq 0.
\end{cases}
\end{equation}

with $S(t)+I(t)+R(t)=N,$\\

where $\nu$ is the infection rate and $\delta$ is the recovery rate.\\

In this study, we will work with proportional quantities instead of the actual populations by scaling each state variable ($S,\  I,\ R$) by the total population. For this purpose, we introduce new state variables ($s,\  i,\ r$) which are expressed in terms of the original state variables and the total population given as follows;

\begin{equation}\label{anthonySCALING}
\begin{cases}
s(t)=\frac{S(t)}{N},\\ \\
i(t)=\frac{I(t)}{N},\\ \\
r(t)=\frac{R(t)}{N}.
\end{cases}
\end{equation}

Differentiating the scaling equation~(\ref{anthonySCALING}) with respect to time $t$ and using the unscaled model problem~(\ref{anthony1}), we obtain the scaled SIR model describing the transmission dynamics of Ebola is as follows

\begin{equation}\label{anthony1erictoday}
\begin{cases}
\frac{ds(t)}{dt}={-\nu s(t)i(t)}, \qquad \qquad s(0) = s_0\geq 0, &\\ \\
\frac{di(t)}{dt} =  {\nu s(t)i(t)}-{\delta i(t)}, \qquad  i(0) = i_0\geq 0,\\ \\
\frac{dr(t)}{dt} = {\delta i(t)}, \ \ \qquad \qquad  \qquad r(0) = r_0\geq 0.
\end{cases}
\end{equation}

with $s(t)+i(t)+r(t)=1$,\\

where the new state variables $s(t),\ i(t)$ and $r(t)$ represent proportions of Susceptible, Infected and Recovered individuals respectively.\\

As in \citep{rachah2015mathematical}, we will use this scaled SIR model~(\ref{anthony1erictoday}) in subsections~\ref{sty1} and \ref{sty2} to formulate optimal control models for Ebola infection with two different optimal control strategies.

\subsection{SIR Model with Optimal Control Strategy 1}\label{sty1}
In this subsection, we consider a controlled dynamical system for the scaled SIR model~(\ref{anthony1erictoday}), using vaccination as time dependent control function, $\eta(t)$. For this optimal control problem, we want to reduce the number of infected individuals and the cost of vaccination. Therefore, in this control strategy, we minimize the objective functional $J(\eta)$ given by

\begin{equation}\label{tony}
 J(\eta)=\int^{t_{f}}_{0}\left[i(t)+\dfrac{B}{2}\eta^{2}(t)\right]dt
\end{equation}

subject to:

\begin{equation}\label{anthony2}
\begin{cases}
\frac{ds(t)}{dt}={-\nu s(t)i(t)- \eta (t)s(t)},\  \qquad s(0) = s_0\geq 0, \\ \\
\frac{di(t)}{dt} =  {\nu s(t)s(t)-\delta i(t)},\ \qquad  \qquad  i(0) = i_0\geq 0,\\ \\
\frac{dr(t)}{dt} = {\delta i(t)+\eta (t)s(t)},\ \ \qquad  \qquad r(0) = r_0\geq 0.
\end{cases}
\end{equation}

where the control set is given as:

\[{\mathcal{G}}_1=  \{\eta:\eta (t)\  \textrm{is lebesgue measurable},\ 0\leq \eta (t)\leq 1,\ t\in[0, t_{f}]\}\]

and the positive parameter $A$ is the weight on cost of vaccination. \\

The Hamiltonian function corresponding to the objective functional~(\ref{tony}) and the dynamic constrained state system~(\ref{anthony2}) is given by
\begin{equation}\label{hmtmldorcas}
\begin{split}
H & = i(t)+\dfrac{B}{2}\eta^{2}(t)\\\\
&+\varphi_{1} \left[- \nu s(t)i(t) - \eta (t)s(t)\right]\\ \\
 & +\varphi_{2}\left[\nu s(t)i(t) - \delta i(t)\right]\\ \\
 & +\varphi_{3}\left[\delta i(t)+\eta (t)s(t)\right]
\end{split}
\end{equation}

We then apply the Maximum Principle proposed by authors in \citep{pontryagin1962mathematical} to determine an optimal solution as follows:\\

Suppose that $(w, \eta)$ is an optimal solution for a controlled dynamical system, then there exist adjoint vector function $\varphi=(\varphi_1,\ \varphi_2,\ \cdots \varphi_n)$ which satisfy the system below;

\begin{equation}\label{anthony2a}
\begin{cases}
\frac{dw}{dt}={\frac{\partial H(t, w, \eta, \varphi)}{\partial \varphi}}, \\ \\
0 =  {\frac{\partial H(t, w, \eta, \varphi)}{\partial \eta}},\\ \\
\frac{d\varphi}{dt} = -{\frac{\partial H(t, w, \eta, \varphi)}{\partial w}}.
\end{cases}
\end{equation}

%By using the Hamiltonian function~(\ref{hmtmldorcas}) and the necessary conditions, we obtain the adjoint system and the control characterization as follows\\
Following the constructed Hamiltonian function~(\ref{hmtmldorcas}) and equation~(\ref{anthony2a}), we present the adjoint equations and the control characterisation as follows

\begin{theorem}\label{vicky1}
	Let ${\eta}^{*}$ be an optimal control and optimal state solutions $s^{*}, \ i^{*}, \ r^{*}$ of the corresponding controlled dynamical system~(\ref{tony})-(\ref{anthony2}) that minimize $J(\eta)$ over ${\mathcal{G}}_1$. Then there exist adjoint variables $\varphi_i$ \ for  $i=1,\ 2, 3$ which satisfy the system below
	
\begin{equation}
	\begin{cases}
	\displaystyle \frac{d \varphi_{1}}{dt} = \displaystyle {\nu i^* (t)(\varphi_{1}-\varphi_{2})+\eta^{*}(t) (\varphi_{1}-\varphi_{3})},\qquad  \\ \\ %[15pt]
	\displaystyle \frac{d \varphi_{2}}{dt} = \displaystyle {-1+\nu s^* (t)(\varphi_{1}-\varphi_{2})+\delta(\varphi_{2}-\varphi_{3})}, \qquad \\  \\ %[10pt]
	\displaystyle  \frac{d \varphi_{3}}{dt} = \displaystyle {0}.
	\end{cases}
	\end{equation}
	
with transversality conditions $\varphi_1(t_{f})=0,\ \varphi_2(t_{f})=0,\ \varphi_3(t_{f})=0 $

	and the control ${\eta}^{*}$ satisfies the optimality condition.
	
\begin{equation}\label{func}
\displaystyle	{\eta}^{*}(t)=min\left\{max\left\{0, \dfrac{s^* (t)}{B}(\varphi_{1}-\varphi_3)\right\}, 1\right\}
	\end{equation}
\end{theorem}

%$\displaystyle	u^{*}(t)=min\left\{1,max\{0, \dfrac{S^*}{B}(\varphi _{1}-\varphi _3)\}\right\}$

\begin{proof}
To derive the system of equations describing the adjoint variables and transversality conditions, we apply the Pontryagin's Maximum Principle and the Hamiltonian function~(\ref{hmtmldorcas}) as follows
	\begin{equation}\label{anthony2k}
	\begin{cases}
	\frac{d\varphi_{1}}{dt}=-\frac{\partial H}{\partial s}, \\ \\
	\frac{d\varphi_{2}}{dt}=-\frac{\partial H}{\partial i},\\ \\
	\frac{d\varphi_{3}}{dt}=-\frac{\partial H}{\partial r}.
	\end{cases}
	\end{equation}

with \begin{equation}\label{anthony3kd}
	\varphi_{i}(t_f)=0, \ \  i=1,\ 2,\ 3.
	\end{equation}

By using the property of the control space ${\mathcal{G}}_1$ and solving the differential equation given by
\begin{equation}\label{vickyabena}
	\frac{\partial H}{\partial \eta}=0
	\end{equation}

on the interior of the control set, the characterization of the optimal control is obtained as given  by equation~(\ref{func}).
\end{proof}

\subsection{SIR Model with Optimal Control Strategy 2}\label{sty2}	
In this subsection, we present a controlled dynamical system by incorporating treatment control, $\eta_1 (t)$ and educational campaign control, $\eta_2 (t)$ into the scaled SIR model~(\ref{anthony1erictoday}). For this optimal control problem, we want to reduce the number of infected individuals, the cost of educational campaign and cost of treatment. Therefore, in this control strategy, we minimise the objective functional given as

\begin{equation}\label{tony1a}
J(\eta_{1}, \eta_{2})=\int^{t_{f}}_{0}\left[C_1 i(t)+\dfrac{C_{2}}{2}\eta^{2}_{1}(t)+\dfrac{C_{3}}{2}\eta^{2}_{2}(t)\right]dt
\end{equation}
subject to

\begin{equation}\label{anthony3}
\begin{cases}
\frac{ds(t)}{dt}={-\nu s(t)i(t)- \eta_2 (t)s(t)},\qquad \qquad \qquad s(0) = s_0\geq 0, &\\ \\
\frac{di(t)}{dt} =  {\nu s(t)i(t)}-{\delta i(t)-\eta_1 (t)i(t)}, \qquad  \qquad  i(0) = i_0\geq 0,\\ \\
\frac{dr(t)}{dt} = {\delta i(t)+\eta_1 (t)i(t)+\eta_2 (t)s(t)} \qquad  \qquad r(0) = r_0\geq 0.
\end{cases}
\end{equation}

where the control set is given by

\[{\mathcal{G}}_2=  \{(\eta_1,\eta_2):\eta_i(t)\  \textrm{is lebesgue measurable},\ 0\leq \eta_i(t)\leq 1,\ t\in[0, t_{f}], \ \textrm{for} \ i=1,2\}\]

 and $C_1$ is a positive constant to keep a balance in the size of $i(t)$. The positive parameters  $C_2$ and  $C_3 $ are the weight on cost for treatments and educational campaigns. \\

The Hamiltonian function corresponding to equations~(\ref{tony1a}) and (\ref{anthony3}) is given as

\begin{equation}\label{hmtmlantony}
\begin{split}
H & = C_1 i(t)+\dfrac{C_{2}}{2}\eta^{2}_{1}(t)+\dfrac{C_{3}}{2}\eta^{2}_{2}(t)\\ \\
 & +\varphi_{1} \left[- \nu s(t)i(t) - \eta_{2}(t)s(t)\right]\\ \\
 & +\varphi_{2}\left[\nu s(t)i(t)- \delta i(t) -\eta_{1}(t) i(t)\right]\\ \\
 & + \varphi_{3}\left[\delta i(t) +\eta_{1}(t)i(t)+\eta_{2}(t)s(t)\right]
\end{split}
\end{equation}

We then apply the Maximum Principle proposed by authors in \citep{pontryagin1962mathematical} to determine an optimal solution as\\

Assume that $(w, \eta)$ is an optimal solution for a controlled dynamical system, then there exist adjoint vector function $\varphi=(\varphi_1,\ \varphi_2,\ \cdots \varphi_n)$ which satisfy the following equations

%If $(w, \eta)$ is an optimal solution of an optimal control, then there exist a non-trivial vector function $\varphi=(\varphi_1,\ \varphi_2,\ \cdots \varphi_n)$ satisfying the following conditions
\begin{equation}\label{anthony3avicky}
\begin{cases}
\frac{dw}{dt}={\frac{\partial H(t, w, \eta, \varphi)}{\partial \varphi}}, \\ \\
0 =  {\frac{\partial H(t, w,\eta, \varphi)}{\partial \eta}},\\ \\
\frac{d\varphi}{dt} = -{\frac{\partial H(t, w, \eta, \varphi)}{\partial w}}.
\end{cases}
\end{equation}

Using the formulated Hamiltonian function~(\ref{hmtmldorcas}) and equation~(\ref{anthony3avicky}), the adjoint equations and the control characterisation are presented in the following Theorem.

\begin{theorem}\label{vicky2}
Let $\eta^{*}_{1}$  and $\eta^{*}_{2}$  be optimal control pair and optimal state solutions $s^{*}, \ i^{*}, \ r^{*}$ of the corresponding controlled dynamical system~(\ref{tony1a})-(\ref{anthony3}) that minimize $J(\eta_1, \eta_2)$ over ${\mathcal{G}}_2$. Then there exist adjoint variables $\varphi_i$ \ for  $i=1,\ 2, 3$ satisfying
	
	\begin{equation}
	\begin{cases}
	\displaystyle \frac{d \varphi_{1}}{dt} = \displaystyle {\nu i^* (t)(\varphi_{1}-\varphi_{2})+\eta^{*}_{2}(t) (\varphi_{1}-\varphi_{3})},\qquad  \\ \\
	\displaystyle \frac{d \varphi_{2}}{dt} = \displaystyle {-C_1+\nu s^{*}(t)(\varphi_{1}-\varphi_{2})+\delta(\varphi_{2}-\varphi_{3})+\eta^{*}_{1}(t)(\varphi_{2}-\varphi_{3})}, \qquad \\ \\
	\displaystyle  \frac{d \varphi_{3}}{dt} = \displaystyle {0}.
	\end{cases}
	\end{equation}
	
	with transversality conditions  $\varphi_1(t_{f})=0,\ \varphi_2(t_{f})=0,\ \varphi_3(t_{f})=0 $
	
	and the control  pair $\eta^{*}_{1}$  and $\eta^{*}_{2}$  satisfies the optimality conditions
	
	\begin{equation}\label{funcantony}
	\begin{cases}
\displaystyle	\eta^{*}_{1}(t)=min\left\{max\left\{0, \dfrac{i^* (t)}{C_{2}}(\varphi _{2}-\varphi _3)\right\}, 1\right\}\\ \\
\displaystyle	\eta^{*}_{2}(t)=min\left\{max\left\{0, \dfrac{s^* (t)}{C_{3}}(\varphi _{1}-\varphi _3)\right\}, 1\right\}
	\end{cases}
	\end{equation}
\end{theorem}

\begin{remark}
We remark that, the proof of Theorem~\ref{vicky2} and subsequent Theorems in the next section of this study are similar to that of Theorem~\ref{vicky1} and we therefore omit their proofs.
\end{remark}

%\begin{equation}\label{func}
%\displaystyle	u^{*}(t)=min\left\{max\left\{0, \dfrac{S^*}{B}(\varphi _{1}-\varphi _3)\right\}, 1\right\}
%	\end{equation}
%\begin{proof}
%	To derive the adjoint equations and transversality conditions we use the Hamiltonian function~(\ref{hmtmlantony}).
%	The adjoint equations is obtained from the Pontryagin's maximum principle
%	\begin{equation}\label{anthony3k}
%	\begin{cases}
%	\frac{d\varphi_{1}}{dt}=-\frac{\partial H}{\partial s}, \\ \\
%	\frac{d\varphi_{2}}{dt}=-\frac{\partial H}{\partial i},\\ \\
%	\frac{d\varphi_{3}}{dt}=-\frac{\partial H}{\partial r}.
%	\end{cases}
%	\end{equation}
%	
%with \begin{equation}\label{anthony3kd}
%	\varphi_{i}(t_f)=0, \ \  i=1,\ 2,\ 3.
%	\end{equation}
%	
%To get the characterization of the optimal control functions~(\ref{funcantony}), we solve the equation
%
%     \begin{equation}\label{anthony3kdocas}
%	\begin{cases}
%	\frac{\partial H}{\partial \eta_1}=0,\\ \\
%	\frac{\partial H}{\partial \eta_2}=0.
%	\end{cases}
%	\end{equation}
%	
%on the interior of the control set and using the property of the control space, we can derived the characterization~(\ref{funcantony}).
%\end{proof}
%
%

\subsection{Numerical Simulations and Discussions}\label{ericpower}
This subsection deals with numerical solutions of optimal control problems formulated in subsections~\ref{sty1} and \ref{sty2}. Analytical solution of nonlinear system of differential equations with optimal control is a very hard task in mathematical modelling. Therefore, in this study, we apply the efficient forward-backward sweep numerical scheme with fourth-order Runge-Kutta method to solve the optimality systems formulated for the two optimal control strategies. This numerical scheme has extensively been described by the authors in \citep{lenhart2007optimal} in their mathematical modelling textbook which is concern with the applications of differential equations and optimal control theory. For our numerical simulations, we have adapted the same model parameter values ($\nu=0.2,\ \delta=0.1$) as in the work by the authors in \citep{rachah2015mathematical} and initial conditions: $s_0=0.95,\ i_0=0.05,\ r_0=0$. For the positive parameters in the objectives functionals, we have assumed that $D=1,\ C_1=1,\ C_2=5$ and $C_3=5.$ In Fig.~\ref{fg2}, it is clear that, there is a significant decrease in the number of susceptible individuals with control strategies 1, 2 than without control. Figure~\ref{fg3} shows solution paths for infected individuals with control strategies 1, 2 and without control. This plot shows significant decrease in the infected individuals with control strategies 1, 2 than without control. In Fig~\ref{fg4}, there is a rapid increase in recovered individuals with control strategies 1, 2 than without control. Figures~\ref{fg1a} and \ref{fg1b} represents control functions for optimal control strategies 1 and 2 respectively.

\begin{figure}[!htbp]
	\centering
	\includegraphics[scale=0.7]{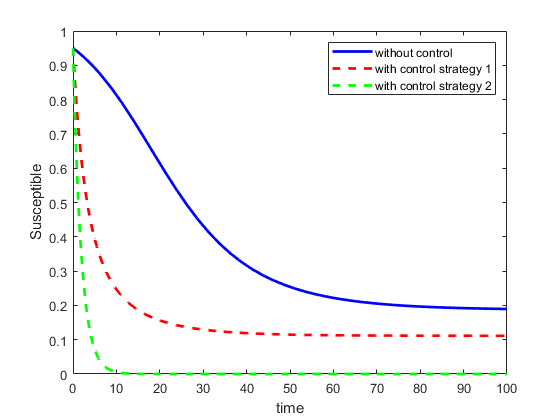}
	\caption{Solution paths for susceptible individuals with two control strategies and without control, $\nu=0.2$, \ $\delta=0.1$}
	\label{fg2}
\end{figure}

\begin{figure}[!htbp]
	\centering
	\includegraphics[scale=0.7]{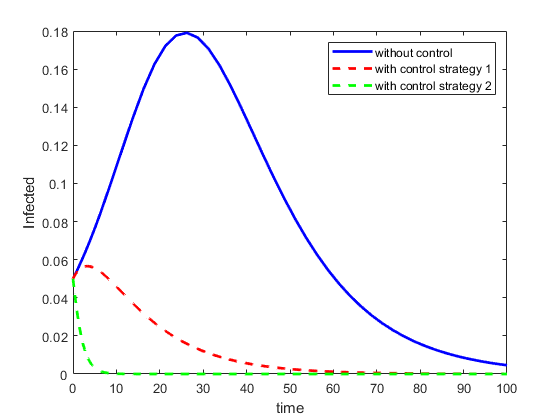}
	\caption{Solution paths for Infected individuals with two control strategies and without control, $\nu=0.2$, \ $\delta=0.1$}
	\label{fg3}
\end{figure}

\begin{figure}[!htbp]
	\centering
	\includegraphics[scale=0.7]{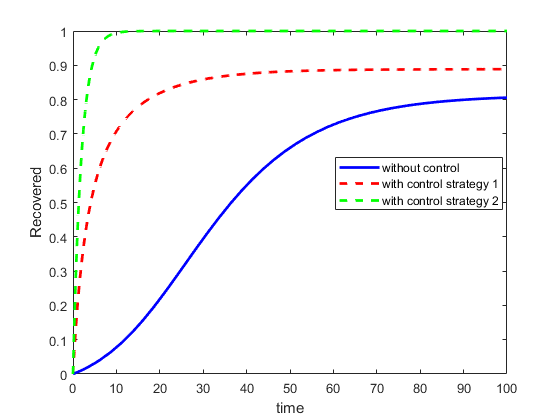}
	\caption{Solution paths for Recovered individuals with two control strategies and without control, $\nu=0.2$, \ $\delta=0.1$}
	\label{fg4}
\end{figure}

\newpage
\begin{figure}[!htbp]
	\centering
	\includegraphics[scale=0.7]{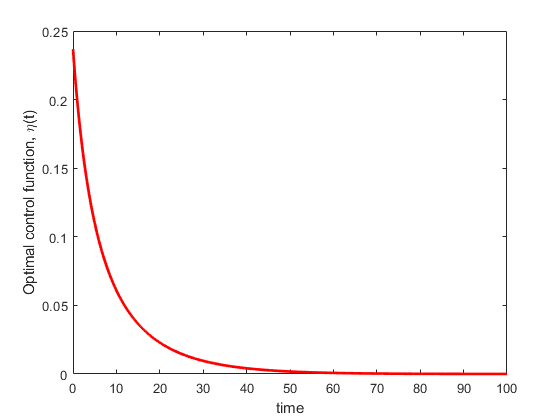}
	\caption{Optimal control function $\eta (t)$ for strategy 1}
	\label{fg1a}
\end{figure}

\begin{figure}[!htbp]
	\centering
	\includegraphics[scale=0.7]{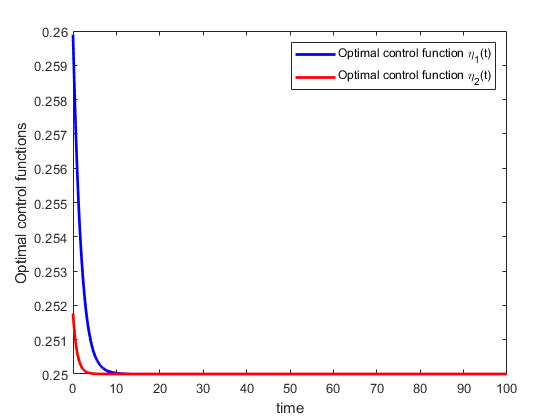}
	\caption{Optimal control functions $\eta_1 (t)$ and $\eta_2 (t)$ for strategy 2}
	\label{fg1b}
\end{figure}
\newpage

\section{SEIR Model}\label{abena779}
In this section, we introduce the basic SEIR mathematical model to describe the transmission dynamics of Ebola infection. This model also assumes constant population size with no vital dynamics(birth and death rates). The total population is divided into four different classes with $S(t),\ E(t),\ E(t)$ and $R(t)$ representing Susceptible, Exposed, Infectious and Recovered individuals respectively. Following the compartmental modeling concepts in \citep{hethcote2000mathematics}, the nonlinear dynamical system describing Ebola virus disease is given by

\begin{equation}\label{docas1}
\begin{cases}
\displaystyle \frac{dS(t)}{dt} = \displaystyle {-\frac{\nu S(t)I(t)}{N}}, \ \ \ \qquad \qquad \qquad   S(0) = S_0\geq 0, \\ \\
\displaystyle \frac{dE(t)}{dt} = \displaystyle {\frac{\nu S(t)I(t)}{N}}-{\rho}{E(t)},\ \ \ \ \ \qquad  E(0) = E_0\geq 0,\\ \\
\displaystyle \frac{dI(t}{dt} = \displaystyle {\rho E(t)}-{\delta}{I(t)},\ \ \qquad \qquad  \qquad  I(0) = I_0\geq 0,\\ \\
\displaystyle \frac{dR(t)}{dt} = \displaystyle {\delta I(t)},\qquad \qquad \qquad  \qquad  \qquad  R(0) = R_0\geq 0.
\end{cases}
\end{equation}

with $S(t)+E(t)+I(t)+R(t)=N,$\\

where $\nu$ is the transmission rate, $\rho$ represent the infectious rate and $\delta$ is the recovery rate.\\

As we did in section~\ref{eric44}, we will work with proportional quantities instead of the actual populations by scaling each state variable ($S(t),\  E(t),\ I(t),\ R(t)$) by the total population. We therefore introduce new state variables ($s(t),\  e(t),\ i(t),\ r(t)$) which are expressed in terms of the original state variables and the total population is given by\\

\begin{equation}\label{anthonySCALINGnew}
\begin{cases}
s(t)=\frac{S(t)}{N},\\ \\
e(t)=\frac{E(t)}{N},\\ \\
i(t)=\frac{I(t)}{N},\\ \\
r(t)=\frac{R(t)}{N}.
\end{cases}
\end{equation}

By differentiating the scaling equation~(\ref{anthonySCALINGnew}) with respect to time $t$ and using the unscaled model problem~(\ref{docas1}), we obtain the scaled SEIR model as follows

\begin{equation}\label{docas2today}
\begin{cases}
\displaystyle \frac{ds(t)}{dt} = \displaystyle {-\nu s(t)i(t)}, \ \qquad \qquad  \qquad   s(0) = s_0\geq 0, \\ \\
\displaystyle \frac{de(t)}{dt} = \displaystyle {\nu s(t)i(t)}-{\rho}{e(t)},   \qquad \qquad  e(0) = e_0\geq 0,\\ \\
\displaystyle \frac{di(t)}{dt} = \displaystyle {\rho e(t)}-{\delta}{i(t)}, \qquad \qquad  \qquad  i(0) = i_0\geq 0,\\ \\
\displaystyle \frac{dr(r)}{dt} = \displaystyle {\delta i(t)},\ \ \ \qquad \qquad \qquad  \qquad  r(0) = r_0\geq 0.
\end{cases}
\end{equation}

with $s(t)+e(t)+i(t)+r(t)=1,$\\

where the new state variables $s(t)\, e(t),\ i(t)$ and $r(t)$ represent proportions of Susceptible, Infected and Recovered individuals respectively.\\

As in \citep{rachah2016dynamics}, we will use this scaled SEIR model problem~(\ref{docas2today}) in subsections~\ref{stySEIR1},\  \ref{stySEIR2} and \ref{stySEIR3} to construct optimal control models for Ebola infection with three different control strategies.

\subsection{SEIR Model with Optimal Control Strategy 1}\label{stySEIR1}
In this subsection, we consider a controlled dynamical system for the scaled SEIR model~(\ref{docas2today}), using vaccination as time dependent control function, $\eta(t)$. As we did in subsection~\ref{sty1}, our main objective for this optimal control problem is that, we want to reduce the number of infected individuals and the cost of vaccination. Therefore in this control strategy, we minimize the objective functional $J(\eta)$ given by

\begin{equation}\label{Abena1}
J(\eta)= \int^{t_{f}}_{0}\left[i(t)+ \dfrac{D}{2}\eta^{2}(t)\right]dt
\end{equation}
subject to:

\begin{equation}\label{docas2}
\begin{cases}
\displaystyle \frac{ds(t)}{dt} = \displaystyle {-\nu s(t)i(t)-\eta (t)s(t)},\  \qquad \qquad   s(0) = s_0\geq 0, \\ \\
\displaystyle \frac{de(t)}{dt} = \displaystyle {\nu s(t)i(t)}-{\rho}{e(t)}, \qquad \qquad \qquad  e(0) = e_0\geq 0,\\ \\
\displaystyle \frac{di(t)}{dt} = \displaystyle {\rho e(t)}-{\delta}{i(t)},  \qquad \qquad \qquad  \qquad  i(0) = i_0\geq 0,\\ \\
\displaystyle \frac{dr(t)}{dt} = \displaystyle {\delta i(t)+\eta (t)s(t)},\ \ \qquad \qquad  \qquad  r(0) = r_0\geq 0.
\end{cases}
\end{equation}

where the control set is given by

\[{\mathcal{G}}_3=  \{\eta:\eta (t)\  \textrm{is lebesgue measurable},\ 0\leq \eta (t)\leq 1,\ t\in[0, t_{f}]\}\]
and the positive constant $D$ is the weight on the cost of vaccination.\\

The Hamiltonian function $H$ is then given as:

\begin{equation} \label{hmtmlantony34}
\begin{split}
H & =i(t)+\dfrac{D}{2}\eta^{2}(t)\\ \\
  & +\varphi_{1} \left[- \nu s(t)i(t) - \eta (t)s(t)\right]\\ \\
 &  +\varphi_{2}\left[\nu s(t)i(t) - \rho e(t)\right]\\ \\
 & +\varphi_{3}\left[\rho e(t)-\delta i(t)\right]\\ \\
 & +\varphi_{4}\left[ \delta i(t)+\eta (t)s(t)\right]
\end{split}
\end{equation}

We then apply the Maximum Principle proposed by authors in \citep{pontryagin1962mathematical} to determine an optimal solution as\\
%In order to determine the necessary conditions, we use the Pontryagin's maximum principle \citep{pontryagin1962mathematical} as follows;\\

Given that $(w, \eta)$ is an optimal solution for a controlled dynamical system, then there exist adjoint vector function $\varphi=(\varphi_1,\ \varphi_2,\ \cdots \varphi_n)$ which satisfy the following equations

%If $(x, \eta)$ is an optimal solution of an optimal control, then there exist a non-trivial vector function $\varphi=(\varphi_1,\ \varphi_2,\ \cdots \varphi_n)$ satisfying the following conditions

\begin{equation}\label{docas2aparty}
\begin{cases}
\frac{dw}{dt}={\frac{\partial H(t, w, \eta, \varphi)}{\partial \varphi}}, \\ \\
0 =  {\frac{\partial H(t, w, \eta, \varphi)}{\partial \eta}},\\ \\
\frac{d\varphi}{dt} = -{\frac{\partial H(t, w, \eta, \varphi)}{\partial w}}.
\end{cases}
\end{equation}

Using the formulated Hamiltonian function~(\ref{hmtmlantony34}) and equation~(\ref{docas2aparty}), the adjoint equations and the control characterisation are presented in the following Theorem.

\begin{theorem}\label{vicky3}
Let ${\eta}^{*}$ be an optimal control and optimal state solutions $s^{*}, e^{*}, i^{*}, r^{*}$ of the corresponding controlled dynamical system~(\ref{Abena1})-(\ref{docas2}) that minimize $J(\eta)$ over ${\mathcal{G}}_3$. Then there exist adjoint variables $\varphi_i$ \ for $i=1, 2, 3, 4$ satisfying

	\begin{equation}
	\begin{cases}
	\displaystyle \frac{d \varphi_{1}}{dt} = \displaystyle {\nu i^* (t)(\varphi_{1}-\varphi{2})+\eta^{*}(t)(\varphi _{1}-\varphi _{4})},\qquad  \\ \\ %[5pt]
	\displaystyle \frac{d \varphi_{2}}{dt} = \displaystyle {\rho (\varphi _2 -\varphi _3)}, \qquad \\ \\%[5pt]
	\displaystyle  \frac{d \varphi_{3}}{dt} = \displaystyle {-1 +\nu s^* (t) (\varphi_{1}-\varphi _{2})+ \delta (\varphi _{3}-\varphi _4)},\qquad \\ \\ %[5pt]
	\displaystyle  \frac{d \varphi_{4}}{dt} = \displaystyle {0}.\qquad
	\end{cases}
	\end{equation}
	
with transversality conditions $\varphi_1(t_{f})=0,\ \varphi_2(t_{f})=0,\ \varphi_3(t_{f})=0,\ \varphi_4(t_{f})=0 $
and the control $\eta^{*}(t)$ satisfies the optimality condition.

\begin{equation}\label{func4a}
\displaystyle	\eta^{*}(t)=min\left\{max\left\{0, \dfrac{s^*}{D}(\varphi _{1}-\varphi _4)\right\}, 1\right\}
	\end{equation}
\end{theorem}

%\begin{proof}
%	To derive the adjoint equations and transversality conditions we use the Hamiltonian function~(\ref{hmtmlantony34}).
%	The adjoint equations is obtained from the Pontryagin's maximum principle
%	\begin{equation}\label{docas2k}
%	\begin{cases}
%	\frac{d\varphi_{1}}{dt}=-\frac{\partial H}{\partial S} \\ \\
%	\frac{d\varphi_{2}}{dt}=-\frac{\partial H}{\partial E}\\ \\
%	\frac{d\varphi_{3}}{dt}=-\frac{\partial H}{\partial I}\\ \\
%	\frac{d\varphi_{4}}{dt}=-\frac{\partial H}{\partial R}
%	\end{cases}
%	\end{equation}
%	
%with \begin{equation}\label{anthony3kd}
%	\varphi_{i}(t_{end})=0, \ \  i=1,\ 2,\ 3,\ 4.
%	\end{equation}
%
%	To get the characterization of the optimal control function~(\ref{func4a}), we solve the equation
%	\begin{equation}
%	\frac{\partial H}{\partial u}=0
%	\end{equation}
%	on the interior of the control set and using the property of the control space, we can derived the characterization~(\ref{func4a}).
%\end{proof}

\subsection{SEIR Model with Optimal Control Strategy 2}\label{stySEIR2}
In this subsection, we present a controlled dynamical system for the scaled SEIR model~(\ref{docas2today}), using vaccination as time dependent control function, $\eta(t)$. Our main objective for this strategy is to reduce number of infected and exposed individuals as were as the cost of vaccination. Therefore in this control strategy, we minimize the objective functional $J(\eta)$ given by

\begin{equation}\label{Abena21}
J(\eta)=\int^{t_{f}}_{0}\left[K_{1}e(t)+ K_{2}i(t)+\dfrac{K_{3}}{2}\eta^{2}(t)\right]dt
\end{equation}

subject to:

\begin{equation}\label{docas4}
\begin{cases}
\displaystyle \frac{ds(t)}{dt} = \displaystyle {-\nu s(t)i(t)-\eta (t)s(t)},\ \qquad \qquad   s(0) = s_0\geq 0, \\ \\
\displaystyle \frac{de(t)}{dt} = \displaystyle {\nu s(t)i(t)}-{\rho}{e(t)},\ \ \ \ \ \ \qquad \qquad  e(0) = e_0\geq 0,\\ \\
\displaystyle \frac{di(t)}{dt} = \displaystyle {\rho e(t)}-{\delta}{i(t)}, \qquad \qquad \qquad  \qquad  i(0) = i_0\geq 0,\\ \\
\displaystyle \frac{dr(t)}{dt} = \displaystyle {\delta i(t)+\eta (t)s(t)}, \ \ \qquad \qquad  \qquad  r(0) = r_0\geq 0.
\end{cases}
\end{equation}

where the control set is given as:

\[{\mathcal{G}}_4=  \{\eta:\eta(t)\  \textrm{is lebesgue measurable},\ 0\leq \eta(t)\leq 1,\ t\in[0, t_{f}]\}\]

and the parameters $K_1$ and $K_2$ are positive constants associated with exposed individuals and infected individuals and $K_3$ is a positive weight constant for control function $\eta (t)$.\\

The Hamiltonian $H$ is given by
\begin{equation} \label{eq1}
\begin{split}
H & =K_{1}e(t)+K_{2}i(t)+\dfrac{K_3}{2}\eta^{2}(t) \\ \\
 & +\varphi _{1}\left[-\nu s(t)i(t)-\eta(t) s(t)\right]\\ \\
 & +\varphi _{2}\left[\nu s(t)i(t)-\rho e(t)\right]\\ \\
 & + \varphi_{3}\left[\rho e(t)-\delta i(t)\right]\\ \\
 &+ \varphi_{4} \left[\delta i+\eta(t)s(t)\right]
\end{split}
\end{equation}

We then apply the Maximum Principle proposed by authors in \citep{pontryagin1962mathematical} to determine an optimal solution as\\

Assume that $(w, \eta)$ is an optimal solution for a controlled dynamical system, then there exist adjoint vector function $\varphi=(\varphi_1,\ \varphi_2,\ \cdots \varphi_n)$ which satisfy the following equations

\begin{equation}\label{docas4a}
\begin{cases}
\frac{dw}{dt}={\frac{\partial H(t, w, \eta, \varphi)}{\partial \varphi}}, \\ \\
0 =  {\frac{\partial H(t, x, \eta, \varphi)}{\partial \eta}},\\ \\
\frac{d\varphi}{dt} = -{\frac{\partial H(t, w, \eta, \varphi)}{\partial w}}.
\end{cases}
\end{equation}

Using the formulated Hamiltonian function~(\ref{eq1}) and equation~(\ref{docas4a}), the adjoint equations and the control characterisation are given in the following Theorem as

\begin{theorem}\label{vicky4}
Let ${\eta}^{*}$ be an optimal control and optimal state solutions $s^{*}, e^{*}, i^{*}, r^{*}$ of the corresponding controlled dynamical system~(\ref{Abena21})-(\ref{docas4}) that minimize $J(\eta)$ over ${\mathcal{G}}_3$. Then there exist adjoint variables $\varphi_i$ \ for $i=1, 2, 3, 4$ satisfying

	\begin{equation}
	\begin{cases}
	\displaystyle \frac{d \varphi_{1}}{dt} = \displaystyle {\nu i^* (t)(\varphi_{1}-\varphi{2})+\eta^{*}(t)(\varphi _{1}-\varphi _{4})},\qquad  \\ \\ %[5pt]
	\displaystyle \frac{d \varphi_{2}}{dt} = \displaystyle {-K_{1}+\rho(\varphi_{2}-\varphi_{3})}, \qquad \\ \\
	\displaystyle  \frac{d \varphi_{3}}{dt} = \displaystyle {-K_{2}+\nu s^* (t)(\varphi _{1}-\varphi _{2})+ \delta (\varphi _{3}-\varphi _{4})},\qquad \\ \\
	\displaystyle  \frac{d \varphi_{4}}{dt} = \displaystyle {0}.
	\end{cases}
	\end{equation}
	with transversality conditions $\varphi_1(t_{f})=0,\ \varphi_2(t_{f})=0,\ \varphi_3(t_{f})=0, \ \varphi_4(t_{f})=0 $
	and the control function $\eta^{*}$ is given by
	\begin{equation}\label{func4b}
	\displaystyle \eta^{*}(t)=min\left\{max\left\{0, \dfrac{s^* (t)}{K_3}(\varphi _{1}-\varphi _4)\right\}, 1\right\}
	\end{equation}
\end{theorem}

\subsection{SEIR Model with Optimal Control Strategy 3}\label{stySEIR3}
In this subsection, we present a controlled dynamical problem by incorporating treatment control, $\eta_1 (t)$ and educational campaign control, $\eta_2 (t)$  into the SEIR model~(\ref{anthony1erictoday}). As we presented in subsection~\ref{sty2}, our main aim for this controlled problem is to minimize the number of infected individuals, the cost of educational campaigns and treatment. Therefore in this control strategy, we minimise the objective functional given by

\begin{equation}\label{Abena3}
J(\eta_{1}, \eta_{2})= \int^{t_{f}}_{0}\left[D_1 i(t)]+\dfrac{D_{2}}{2}\eta^{2}_{1}(t)+\dfrac{D_{3}}{2}\eta^{2}_{2}(t)\right]dt
\end{equation}

Subject to:

\begin{equation}\label{docas6anthony}
\begin{cases}
\displaystyle \frac{ds(t)}{dt} = \displaystyle {-\nu s(t)i(t)-\eta_2 (t)s(t)}, \ \ \ \ \qquad \qquad \qquad   s(0) = s_0\geq 0, \\ \\
\displaystyle \frac{de(t)}{dt} = \displaystyle {\nu s(t)i(t)}-{\rho}{e(t)}, \ \ \ \ \ \ \qquad \qquad \qquad \qquad  e(0) = e_0\geq 0,\\ \\
\displaystyle \frac{di(t)}{dt} = \displaystyle {\rho e(t)}-{\delta}{i(t)}-\eta_{1}(t)i(t), \ \ \qquad \qquad  \qquad  i(0) = i_0\geq 0,\\ \\
\displaystyle \frac{dr(t)}{dt} = \displaystyle {\delta i(t)+\eta_{1}(t)i(t)+\eta_2(t)s(t)},\ \ \ \ \qquad  \qquad  r(0) = r_0\geq 0.
\end{cases}
\end{equation}

where the control set is given as:

\[{\mathcal{G}}_5=  \{(\eta_1, \eta_2):\eta_i(t)\  \textrm{is lebesgue measurable},\ 0\leq \eta_i(t)\leq 1,\ t\in[0, t_{f}]\ \textrm{for} \ i=1,2\}\]

 where $D_1$ is the weight constant on infected individuals and $ D_2$ and $D_3$ are also positive weight parameters associated with controls $\eta_1 (t)$ and $\eta_2 (t)$ respectively.\\

The Hamiltonian function $H$ is given as
\begin{equation}\label{hmtmlkwasi}
\begin{split}
H& =D_1 i(t)]+\dfrac{D_{2}}{2}\eta^{2}_{1}(t)+\dfrac{D_{3}}{2}\eta^{2}_{2}(t)\\ \\
& +\varphi _{1}\left[-\nu s(t)i(t)-\eta_{2}(t)s(t)\right]\\ \\
& +\varphi_{2} \left[\nu s(t)i(t)-\rho e(t)\right]\\ \\
& +\varphi_{3}\left[\rho e(t)-\delta i(t)-\eta_{1}(t)i(t)\right] \\ \\
& +\varphi_{4}\left[\delta i(t)+\eta_{1}(t)i(t)+\eta_{2}(t)s(t)\right]
\end{split}
\end{equation}

We then apply the Maximum Principle proposed by authors in \citep{pontryagin1962mathematical} to determine an optimal solution as\\

Assume that $(w, \eta)$ is an optimal solution for a controlled dynamical system, then there exist adjoint vector function $\varphi=(\varphi_1,\ \varphi_2,\ \cdots \varphi_n)$ which satisfy the following equations

\begin{equation}\label{docas6b}
\begin{cases}
\frac{dw}{dt}={\frac{\partial H(t, w, \eta, \varphi)}{\partial \varphi}}, \\ \\
0 =  {\frac{\partial H(t, w, \eta, \varphi)}{\partial \eta}},\\ \\
\frac{d\varphi}{dt} = -{\frac{\partial H(t, w, \eta, \varphi)}{\partial w}}.
\end{cases}
\end{equation}

Using the formulated Hamiltonian function~(\ref{hmtmlkwasi}) and equation~(\ref{docas6b}), the adjoint equations and the control characterisation are presented in the following Theorem.

\newpage
\begin{theorem}\label{vicky5}
Let $\eta^{*}_{1}$  and $\eta^{*}_{2}$  be optimal control pair and optimal state solutions $s^{*}, e^{*}, i^{*}, r^{*}$ of the corresponding controlled dynamical system~(\ref{Abena3})-(\ref{docas6anthony}) that minimize $J(\eta_1, \eta_2)$ over ${\mathcal{G}}_5$. Then there exist adjoint variables $\varphi_i$ \ for  $i=1, 2, 3, 4$ satisfying

	\begin{equation}
	\begin{cases}
	\displaystyle \frac{d \varphi_{1}}{dt} = \displaystyle {\nu i^* (t) (\varphi_{1}-\varphi_{2})+\eta^{*}_{2}(t)(\varphi_{1}-\varphi_{4})} ,\qquad  \\ \\ %[5pt]
	\displaystyle \frac{d \varphi_{2}}{dt} = \displaystyle {\rho(\varphi_{2}-\varphi_{3})}, \qquad \\ \\ %[5pt]
	\displaystyle  \frac{d \varphi_{3}}{dt} = \displaystyle {-D_1+\nu s^* (t)(\varphi _{1}-\varphi _{2})+ \delta (\varphi _{3}-\varphi _{4})+\eta^{*}_{1} (t)(\varphi _{3}-\varphi _{4})},\qquad \\ \\ %[5pt]
	\displaystyle  \frac{d \varphi_{4}}{dt} = \displaystyle {0}.\qquad
	\end{cases}
	\end{equation}
	
	with transversality conditions $\varphi_1(t_{f})=0,\ \varphi_2(t_{f})=0,\ \varphi_3(t_{f})=0, \ \varphi_4(t_{f})=0 $
	
	and the control pair $\eta^{*}_{1}$  and $\eta^{*}_{2}$ satisfies the optimality condition.
	
\begin{equation}\label{docas6aworkdd}
\begin{cases}
\displaystyle \eta^{*}_{1}(t)=min\left\{max\left\{0, \dfrac{i^* (t)}{D_{2}}(\varphi _{3}-\varphi _{4})\right\}, 1\right\} \\ \\
\displaystyle \eta^{*}_{2}(t)=min\left\{max\left\{0, \dfrac{s^* (t)}{D_{3}}(\varphi _{1}-\varphi_{4})\right\}, 1\right\}
\end{cases}
\end{equation}
\end{theorem}

\subsection{Numerical Simulations and Discussions}
As we did in subsection~\ref{ericpower}, we have used the forward-backward sweep numerical scheme with fourth order Runge-Kutta method to solve the optimality systems formulated for the three optimal control strategies. For our numerical simulations, we have adapted the same model parameter values ($\nu=0.2,\ \delta=0.1, \ \rho=0.1887$) as in the work by the authors in \citep{rachah2016dynamics} with  initial conditions: $s_0=0.88,\ e_0=0.07,\ i_0=0.05,\ r_0=0$.  We have assumed $D=5,\ K_1=1,\ K_2=5,\ K_3=5, \ D_1=1,\ D_2=5$ and $D_3=5$. Figure~\ref{fg5} shows solution paths for susceptible individuals with control strategies $1, 2, 3$ and without control. It is clear from the plot that, there is a significant decrease in the number of susceptible individuals with control strategies than without control. In Fig.~\ref{fg6}, there is a rapid decrease in exposed individuals with control  strategies $1, 2, 3$ than without control. A similar effect can be observed in Fig.~\ref{fg7} for the  infected individuals. In Fig.~\ref{fg8}, there is a rapid increase in recovered individuals with control strategies $1, 2, 3$ than without control. Figures~\ref{fg1bb}, \ \ref{fg1bc}, and \ref{fg1bd} represents optimal control functions for strategies 1,\ 2 and 3 respectively.

\begin{figure}[!htbp]
	\centering
	\includegraphics[scale=0.7]{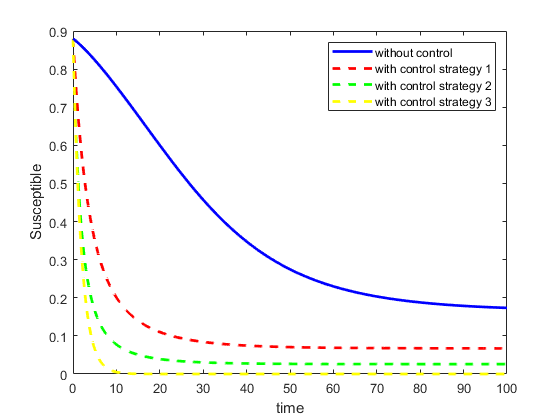}
	\caption{Solution path for Susceptible individuals with three control strategies and without control, $\nu=0.2$, $\rho=0.1887$, \ $\delta=0.1$}
	\label{fg5}
\end{figure}

\begin{figure}[!htbp]
	\centering
	\includegraphics[scale=0.7]{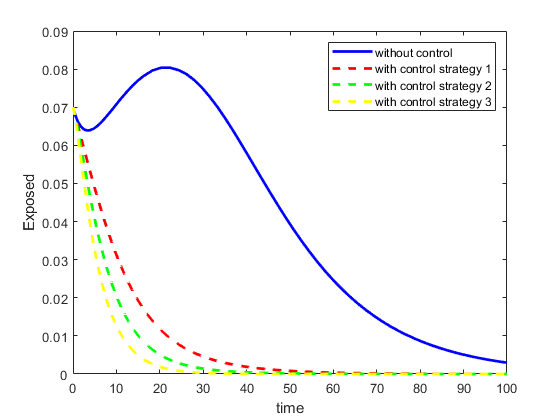}
	\caption{Solution paths for Exposed individuals with three control strategies and without control, $\nu=0.2$, $\rho=0.1887$, \ $\delta=0.1$}
	\label{fg6}
\end{figure}

\begin{figure}[!htbp]
	\centering
	\includegraphics[scale=0.7]{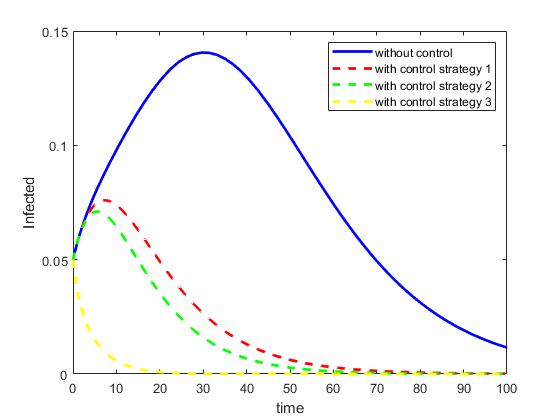}
	\caption{Solution paths for Infected individuals with three control strategies and without control, $\nu=0.2$, $\rho=0.1887$, \ $\delta=0.1$}
	\label{fg7}
\end{figure}

\begin{figure}[!htbp]
	\centering
	\includegraphics[scale=0.7]{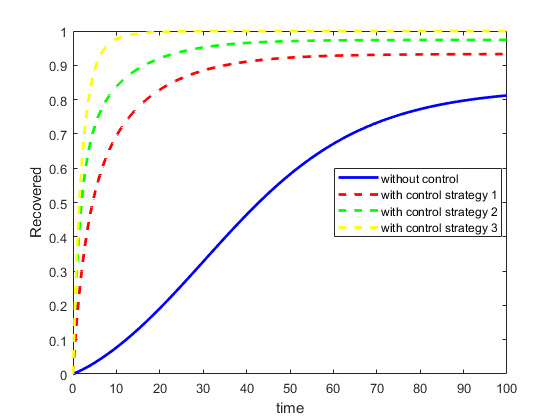}
	\caption{Solution paths for Recovered individuals with three control strategies and without control, $\nu=0.2$, $\rho=0.1887$, \ $\delta=0.1$}
	\label{fg8}
\end{figure}

\begin{figure}[!htbp]
	\centering
	\includegraphics[scale=0.7]{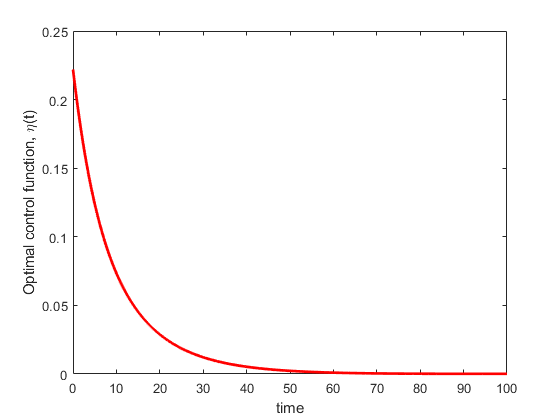}
	\caption{Optimal control function $\eta (t)$ for strategy 1}
	\label{fg1bb}
\end{figure}

\begin{figure}[!htbp]
	\centering
	\includegraphics[scale=0.7]{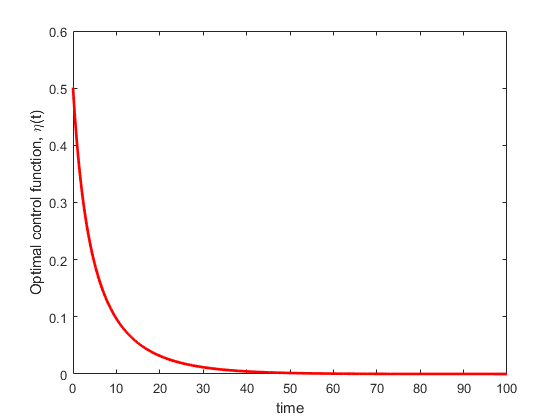}
	\caption{Optimal control functions $\eta (t)$ for strategy 2}
	\label{fg1bc}
\end{figure}

\begin{figure}[!htbp]
	\centering
	\includegraphics[scale=0.7]{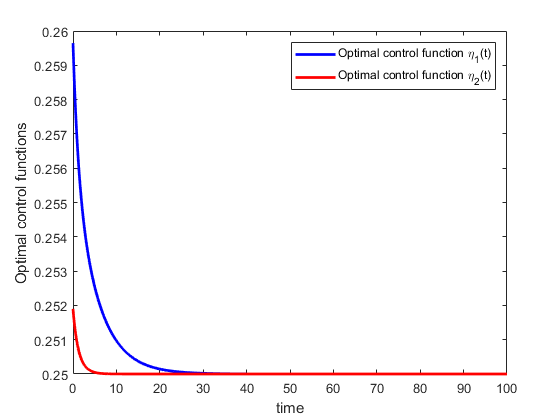}
	\caption{Optimal control functions $\eta_1 (t)$ and $\eta_2 (t)$ for strategy 3}
	\label{fg1bd}
\end{figure}

\newpage
\section{Conclusion}\label{threesec}
In this paper, we have studied epidemiological models for Ebola virus disease using nonlinear system of ordinary differential equation and optimal control theory. We have used indirect methods in optimal control to study existing mathematical models proposed by the authors in \citep{rachah2015mathematical, rachah2016dynamics}. Using the Pontryagin's maximum principle and Hamiltonian function, we have derived adjoint equations and optimality system for the various optimal control strategies. From the simulations results, we observed that, SIR model with optimal control strategies shows significant decrease in the proportions of infected and susceptible individuals and a rapid increase in the recovered individuals compared to SIR model without control. A similar effect was observed in the SEIR model with control strategies where there was a significant decrease in Susceptible, Exposed and Infected individuals and a rapid increase in the Recovered individuals. Our numerical results are similar to the once generated by the authors in \citep{rachah2015mathematical, rachah2016dynamics} who applied direct methods in optimal control for their mathematical models. Therefore, following the numerical results, we can conclude that, effective educational campaigns and vaccination of susceptible individuals as were as effective treatments of infected individuals can help reduced the disease transmission. In this work, we have also analytically and numerical studied one additional optimal control strategy for the SIR epidemic model.

\bibliographystyle{spphys}       % APS-like style for physics

\bibliography{eric}   % name your BibTeX data base

\end{document}